\documentstyle[11pt]{article}
\input{xy}
\xyoption{all}
\newtheorem{th}{Theorem}[section]
\newtheorem{defin}[th]{\bf Definition}
\newtheorem{prop}[th]{\bf Proposition}
\newtheorem{cor}[th]{\bf Corollary}
\newtheorem{lemma}[th]{\bf Lemma}
\newtheorem{ex}[th]{\bf Example}

\newenvironment{proof}{{\bf Proof.}}{\hfill $\Box$\par\vskip3mm}

\title{\bf Characterisation of PF rings by the Finite Topology on duals of $R$ Modules}
\author{\sc Miodrag Cristian Iovanov \small \\  Department of Algebra, Faculty of Mathematics, University of Bucharest \\ Academiei 14, Bucharest, Romania}
\date{}
\begin{document}
\baselineskip16pt
\maketitle
\begin{quotation}\noindent{\small{\sc Abstract}. In this paper we study the properties of the finite topology on the dual of a module over an arbitrary ring. We aim to give conditions when certain properties of the field case are can be still found here. Investigating the correspondence between the closed submodules of the dual $M^{*}$ of a module $M$ and the submodules of $M$, we prove some characterisations of PF rings: the up stated correspondence is an anti isomorphism of lattices iff $R$ is a PF ring.}\end{quotation}\bigskip\bigskip

\begin{section}{Introduction and preliminaries}
Let $R$ be an arbitrary (non commutative) ring. We will use the notations ${\rm Hom}_{R}(M,N)$ for the set of $R$ module morphisms from $M$ to $N$ for right modules $M,N$ and ${}_{R}{\rm Hom}(M,N)$ respectively for left modules $M,N$. Also we use $M^{*}={\rm Hom}_{R}(M,R)$ for any right module $M$ and ${}^{*}M={}_{R}{\rm Hom}(M,R)$ for a left module $M$.\\
 Given two right $R$ modules $M$ and $N$, recall that the finite topology on ${\rm Hom}_{R}(M,N)$ is the linear topology for which a basis of open neighborhoods for $0$ is given by the sets $\{f\in {\rm Hom}_{R}(M,N)\mid f(x_{i})=0,\,\forall\,i\in \{1,\dots,n\}\}$, for all finite sets $\{x_{1},\dots,x_{n}\}\subseteq M$. This is actually the topology induced on ${\rm Hom}_{R}(M,N)$ from ${\rm Hom}_{Set}(M,N)=N^{M}$ which is a product of topological spaces, where $N$ is the topological discrete space on the set $N$. For an arbitrary set $X\subseteq M$ we denote by $X^{\perp}=\{f\in {\rm Hom}_{R}(M,N)\mid f\vert_{X}=0\}$. Denoting by $<X>_{R}$ the $R$ submodule generated by $X$, we obviously have $(<X>_{R})^{\perp}=X^{\perp}$, so we will work with finitely generated submodules $F\leq M$ and the basis of open neighborhoods $\{F^{\perp}\mid F\leq M\,{\rm finitely\,generated}\}$. Also for left $R$ modules $X$ and $Y$ and $U\leq X$ a submodule of $X$ we will denote $U^{\perp}_{{}_{R}{\rm Hom}(M,N)}$ or simply $U^{\perp}=\{g\in {}_{R}{\rm Hom}(X,Y)\mid g\vert_{X}=0\}$ when there is no danger of confusion. If $W\leq {\rm Hom}_{R}(M,N)$ is a subgroup with $M$ and $N$ left $R$ modules we denote $W^{\perp}=\{x\in N\mid f(x)=0,\,\forall\,f\in W\}$. If $N$ is an $R$ bimodule then we consider the left $R$ module structure on ${\rm Hom}_{R}(M,N)$ given by $(r\cdot f)(x)=rf(x)$, for all $x\in M,\,f\in{\rm Hom}_{R}(M,N),\,r\in R$. If $W$ is a (left) submodule in ${\rm Hom}_{R}(M,N)$, then $W^{\perp}$ is a (right) submodule of M.\\
For any right module $M$ we denote by $\Phi_{M}$ the right $R$ modules morphism $$M\stackrel{\Phi_{M}}{\longrightarrow}{}^{*}(M^{*})$$ 
defined by $\Phi_{M}(m)(f)=f(m)$, for all $f\in M^{*}$ and all $m\in M$. Then $\Phi$ is a functorial morphism from $id_{{\cal M}_{R}}$ to the functor ${}^{*}((-)^{*})$.\\
Over a field, there is a series of properties involving the orthogonal $F^{\perp}$ for a vector space $V$ and its dual $V^{*}$ which we will state in a more general setting.

\begin{prop}\label{1}
Let $M,N$ be $R$ modules.\\
(i) If $X\subseteq Y$ are submodules of $M$ then $Y^{\perp}\leq X^{\perp}$. \\
(ii) If $U\subseteq V$ are subgroups of ${\rm Hom}_{R}(M,N)$ then $V^{\perp}\leq U^{\perp}$.
\end{prop}

\begin{lemma}\label{2}
For $M,N$ right $R$ modules we have:\\
(i) If $X\leq M$ is a submodule of $M$ then $(X^{\perp})^{\perp}\supseteq X$ and if we denote $\overline{0}$ the class of $0$ in $M/X$ then we have $(\{\overline{0}\}^{\perp})^{\perp}=(X^{\perp})^{\perp}/X$. If $N$ is an injective cogenerator of ${\cal M}_{R}$ then the equality $(X^{\perp})^{\perp}=X$ holds.\\
(ii) If $Y\leq {\rm Hom}_{R}(M,N)$ is a (left) submodule of ${\rm Hom}_{R}(M,N)$ then $(Y^{\perp})^{\perp}\supseteq \overline{Y}$ ($\overline{Y}$ is the closure of $Y$ in ${\rm Hom}_{R}(M,N)$). If $N=R$ and $R$ is a left PF ring (${}_{R}R$ is injective and a cogenerator of ${}_{R}{\cal M}$) then the equality $(Y^{\perp})^{\perp}=\overline{Y}$ holds for all modules $M$ and (left) submodules $Y\leq M^{*}$.
\end{lemma}
\begin{proof}
(i) If $x\in X$ then take $f\in X^{\perp}$; then $f(x)=0$ as $f\vert_{X}=0$. We get that $f(x)=0,\,\forall\, f\in X^{\perp}$ so $x\in (X^{\perp})^{\perp}$. Moreover, $\overline{x}\in (\{\overline{0}\}^{\perp})^{\perp}$ if and only if $\tilde{h}(\overline{x})=0,\,\forall\, \tilde{h}:M/X\longrightarrow N$, equivalent to $h(x)=0,\,\forall h\in X^{\perp}$, i.e. $x\in (X^{\perp})^{\perp}$.\\
Suppose now $N$ is an injective cogenerator of ${\cal M}_{R}$ and take $x\in (X^{\perp})^{\perp}$. If $x\notin X$ then there is $f:M/X\longrightarrow N$ such that $f(\hat{x})\neq 0$ ($\hat{x}$ is the image of $x$ in $M/X$ via the canonic morphism $\pi: M\longrightarrow M/X$). Then there is $g=f\circ \pi,\, g\in {\rm Hom}_{R}(M,N)$ such that $g\vert_{X}=0$ ($g\in X^{\perp}$) and $g(x)\neq 0$, showing that $x\notin (X^{\perp})^{\perp}$, a contradiction.\\
(ii) Let $f\in \overline{Y}$ and take $x\in Y^{\perp}$. Then there is $g\in Y$ such that $f(x)=g(x)$. But $g(x)=0$ because $x\in Y^{\perp}$ so $f(x)=0$. Thus $f\vert_{Y^{\perp}}=0$ and $f\in (Y^{\perp})^{\perp}$. \\
For the converse, first we see that ${}_{R}R$ injective implies that for all finitely generated right $R$ modules $F$ we have that $F\stackrel{\Phi_{F}}{\longrightarrow}{{}^{*}(F^{*})}$ 
is an epimorphism. Take $\pi:P=R^{n}\longrightarrow F$ an epimorphism in ${\cal M}_{R}$. Then we have a monomorphism $0\longrightarrow P^{*}\longrightarrow F^{*}$ in ${}_{R}{\cal M}$, and as ${}_{R}R$ is injective we obtain an epimorphism of right modules ${}^{*}(P^{*})\stackrel{{}^{*}(p^{*})}{\longrightarrow}{}^{*}(F^{*})\longrightarrow 0$. Because $\Phi$ is a functorial morphism then we have the commutative diagram
\vspace{.5cm}
$$
\xymatrix{
P \ar[r]^\pi\ar[d]_{\Phi_P} & F \ar[r]\ar[d]^{\Phi_F} & 0\\
{}^{*}(P^{*})\ar[r]_{{}^{*}(\pi^{*})} & {}^{*}(F^{*}) \ar[r] & 0
}
$$
\vspace{.5cm}
showing that $\Phi_{F}$ is surjective, as $\Phi_{P}=\Phi_{R^{n}}$ is an isomorphism. Now to prove the desired equality, take $f\in (Y^{\perp})^{\perp}$, $(f_{i})_{i\in I}$ a family of generators of the left $R$ module $Y$, and $F<M$ a finitely generated submodule of $M$. Then $f_{i}\vert M\in F^{*}$ and if $f\vert F\notin {}_{R}<f_{i}\vert_{F}\mid i\in I>$ then as ${}_{R}R$ is an injective cogenerator of ${}_{R}{\cal M}$ we can find a morphism of left $R$ modules $\phi: F^{*}\longrightarrow R$ such that $\phi(f_{i})=0,\,\forall i\in I$ and $\phi(f)\neq 0$. But as $\Phi_{F}$ is surjective, we can then find $x\in F$ such that $\phi=\Phi(x)$ and then $f_{i}(x)\Phi(x)(f_{i})=\phi(f_{i})=0,\,\forall i\in I$, showing that $x\in Y^{\perp}$ and $f(x)=\Phi(x)(f)=\phi(f)\neq 0$ which contradicts the fact that $f$ belongs to $(Y^{\perp})^{\perp}$. Thus we must have $f\vert_{F}\in {}_{R}<f_{i}\vert_{F}\mid i\in I>$ so there is $(r_{i})_{i\in I}$ a family of finite support such that $f\vert_{F}=\sum\limits_{i\in I}r_{i}(f_{i}\vert_{F})=(\sum\limits_{i\in I}r_{i}f_{i})\vert_{F}$. This last relation shows that $f\in \overline{Y}$.
\end{proof}

\begin{cor}
If $R$ is a PF ring (left and right) then for any right (or left) $R$ module $M$ and $Y<M^{*}$ we have that $Y$ is dense in $M^{*}$ if and only if $Y^{\perp}=0$.
\end{cor}

\begin{prop}
Let $M$ be a right $R$ module. \\
(i) If $X\leq M$ then we have $((X^{\perp})^{\perp})^{\perp}=X^{\perp}$ and $X^{\perp}$ is closed.\\
(ii) If $Y\leq {\rm Hom}_{R}(M,N)$ then $((Y^{\perp})^{\perp})^{\perp}=Y^{\perp}$.
\end{prop}
\begin{proof}
"$\subseteq$" from (i) and (ii) follow from Proposition \ref{1} and Lemma \ref{2}.\\
(i) "$\supseteq$" Let $f\in X^{\perp}$. Take $x\in (X^{\perp})^{\perp}$; then $f(x)=0$ so $f\in ((X^{\perp})^{\perp})^{\perp}$. To show that $X^{\perp}$ is closed take $f\in \overline{X^{\perp}}$ and $x\in X$. Then there is $g\in X^{\perp}$ such that $g(x)=f(x)$ so $f(x)=0$ ($x\in X$). We obtain that $f\vert_{X}=0$ so $f\in X^{\perp}$.\\
(ii) "$\supseteq$" Let $x\in Y^{\perp}$. If $f\in (Y^{\perp})^{\perp}$ then $f\vert_{Y^{\perp}}=0$ so $f(x)=0$ showing that $x\in ((Y^{\perp})^{\perp})^{\perp}$.
\end{proof}

\begin{prop}\label{3}
Let $M,N$ be right $R$ modules and $(X_{i})_{i\in I}$ a family of submodules of $M$. Then \\
(i) $(\sum\limits_{i\in I}X_{i})^{\perp}=\bigcap\limits_{i\in I}X_{i}^{\perp}$. \\
(ii) $(\bigcap\limits_{i\in I}X_{i})^{\perp}\supseteq \sum\limits_{i\in I}X_{i}^{\perp}$. If $I$ is finite and $N$ is injective then equality holds.
\end{prop}
\begin{proof}
(i) $f\in (\sum\limits_{i\in I}X_{i})^{\perp}\Leftrightarrow f\vert_{\sum\limits_{i\in I}X_{i}}=0\Leftrightarrow f\vert_{X_{i}}=0,\,\forall i\in I\Leftrightarrow f\in X_{i}^{\perp},\,\forall i\in I\Leftrightarrow f\in \bigcap\limits_{i\in I}X_{i}^{\perp}$.\\
(ii) "$\supseteq$" is obvious, for Proposition \ref{1} shows that $X_{i}^{\perp}\subseteq {\bigcap\limits_{j\in I}X_{j}}^{\perp},\,\forall i\in I$. For the converse it is enough to prove the equality for two submodules $X,Y$ of $M$. Denote $\pi:M\longrightarrow M/X\cap Y$, $p:M\longrightarrow M/X$, $q:M\longrightarrow M/Y$ the canonical morphisms. If $f\in {\rm Hom}_{R}(M,N)$ such that $f\vert _{X\cap Y}=0$ then denote $\overline{f}:M/X\cap Y\longrightarrow N$ the factorisation of $f$ ($f=\overline{f}\circ \pi$) and $i:M/X\cap Y\longrightarrow M/X\oplus M/Y$ the injection $i(\pi(x))=(p(x),q(x)),\,\forall x\in M$. Then the diagram 
\vspace{.5cm}
$$
\xymatrix{
0 \ar[r] & \frac{M}{X\cap Y} \ar[r]^{i}\ar[d]_{\overline{f}} & \frac{M}{X}\oplus\frac{M}{Y} \ar[dl]^{h=\overline{u}\oplus\overline{v}}\\
 & N & 
}
$$
\vspace{.5cm}
is completed commutatively by $h$. Then $h=\overline{u}\oplus\overline{v}$, with $\overline{u}\in {\rm Hom}_{R}(M/X,N)$ and ${\rm Hom}_{R}(M/Y,N)$, such that $h(p(x),q(x))=\overline{u}(p(x))+\overline{v}(q(x))$. Taking $u=\overline{u}\circ p$ and $v=\overline{v}\circ q$ we have $u\in X^{\perp},\,v\in Y^{\perp}$ and $f(x)=\overline{f}(\pi(x))=h(i(\pi(x)))=h(p(x),q(x))=\overline{u}(p(x))+\overline{v}(q(x))=u(x)+v(x),\,\forall x\in M$, so $f\in X^{\perp}+Y^{\perp}$.
\end{proof}

\begin{prop}\label{4}
Let $M,N$ be right $R$ modules and $(Y_{i})_{i\in I}$ a family of submodules of ${\rm Hom}_{R}(M,N)$. Then: \\
(i) $(\sum\limits_{i\in I}Y_{i})^{\perp}=\bigcap\limits_{i\in I}Y_{i}^{\perp}$. \\
(ii) $(\bigcap\limits_{i\in I}Y_{i})^{\perp}\supseteq\sum\limits_{i\in I}Y_{i}^{\perp}$. If $N=R$ and $R$ is a PF ring (both left and right PF) and $Y_{i}$ are closed subsets of $M^{*}={\rm Hom}_{R}(M,R)$ then the equality holds: $(\bigcap\limits_{i\in I}Y_{i})^{\perp}=\sum\limits_{i\in I}Y_{i}^{\perp}$.
\end{prop}
\begin{proof}
(i) Obvious.\\
(ii) "$\supseteq$" similar to (ii)"$\supseteq$" of the previous proposition. For the converse inclusion, take $(Y_{i})_{i\in I}$ a family of submodules of $M^{*}$. Then
\begin{eqnarray*}
\sum\limits_{i\in I}Y_{i}^{\perp} & = & ((\sum\limits_{i\in I}Y_{i}^{\perp})^{\perp})^{\perp} \;\;\;({\rm from\,Lemma\,}\ref{2}:R {\rm \,is\,right\,PF} )\\
 & = & (\bigcap\limits_{i\in I}(Y_{i}^{\perp})^{\perp})^{\perp} \;\;\;({\rm from\,Proposition\,}\ref{3})\\
 & = & (\bigcap\limits_{i\in I}Y_{i})^{\perp}\;\;\;({\rm Lemma\,}\ref{2}\,:Y_{i}{\rm \,are\,closed\,and\,}R{\rm\,is\,left\,PF})
\end{eqnarray*} 
\end{proof}

\begin{ex}
(i) We show that the equality in Proposition \ref{3} does not hold for infinite sets. Let $V$ be an infinite dimensional space with a countable basis indexed by the set of natural numbers: $(e_{n})_{n\in\mathbf{N}}$.  Put $V_{n}=<e_{k}\mid k\geq n>$. Then we can easily see that $\bigcap\limits_{n\in\mathbf{N}} V_{n}=0$ so $(\bigcap\limits_{n\in\mathbf{N}} V_{n}=0)^{\perp}=V^{*}$. Let $f\in V^{*}$ be the function equal to $1$ on all the $e_{n}$-s. Then as $V_{n}^{\perp}<V_{m}^{\perp},\,\forall\, n<m$, we have that $f\in \sum\limits_{n\in\mathbf{N}}V_{n}^{\perp}\Leftrightarrow\exists\,n\in{\mathbf N}\, such\, that \,f\in V_{n}^{\perp}$ which is impossible as $f(e_{n})=0,\,\forall\, n$. We obtain $\bigcap\limits_{n\in\mathbf{N}} V_{n}\supset\sum\limits_{n\in\mathbf{N}}V_{n}^{\perp}$ a strict inclusion.  
\vspace{.5cm}

\noindent
(ii) We show now that the equality in Proposition \ref{4} does not hold for non-closed sets. Let again $V$ be a vector space with a countable basis $B=(e_{n})_{n\in\mathbf{N}}$. Denote by $e_{n}^{*}$ the linear map equal to $1$ on $e_{n}$ and $0$ on the other elements of the basis $B$ and by $f^{*}$ the linear map equal to $1$ on all the $e_{n}$-s. Take $H=<e_{n}^{*}\mid n\in\mathbf{N}>$ and $L=<f^{*},e_{n}^{*}\mid n\in\mathbf{N}^{*}>$. Then we can easily see that $H^{\perp}=0$, $L^{\perp}=0$ and $H\cap L=<e_{n}^{*}\mid n\in \mathbf{N}^{*}>$, so $H^{\perp}+L^{\perp}=0$, but $(H\cap L)^{\perp}=<f^{*},e_{n}^{*}\mid n\in {\mathbf N}>^{\perp}=<e_{0}>$, thus $H^{\perp}+L^{\perp}\neq(H\cap L)^{\perp}$. 
\vspace{.5cm}

\noindent
(iii) Given the same vector space, we give an example of a family of dense subspaces of $V^{*}$ whose intersection is 0. For $p\in {\mathbf N}$ let $H_{p}=<e_{n}^{*}+e_{n+1}^{*}+\dots+e_{n+p}^{*}\mid n\in {\mathbf N}>$. Then a short computation shows that $H_{n}^{\perp}=0$ showing that $H_{n}$ is closed in $V^{*}$. But $\bigcap\limits_{n\in {\mathbf N}}H_{n}=0$, because if $f=\sum\limits_{i=1}^{m}\lambda_{i}e_{i}^{*}\in\bigcap\limits_{n\in {\mathbf N}}H_{n}\subset H_{0}$, then $f\in H_{m+1}$ which shows that if $f\neq 0$, than it can be written as a linear combination of $e_{i}^{*}$ in which at least one of the $e_{i}^{*}$ has $i>m$. This is impossible as the $e_{n}^{*}$-s are independent.
\end{ex}

\end{section}

\newpage

\begin{section}{The Finite Topology vs PF Rings}
If $R$ is a ring then we have $(R^{n})^{*}={\rm Hom}_{R}(R,R)\simeq {}_{R}R^{n}$. So we can identify $R$ submodules of the right dual of $R^{n}$ with left submodules of ${}_{R}R$ and vice versa. For all $x=(x_{1},\dots,x_{n})\in R^{n}$ we denote by $\varphi_{x}:R^{n}\longrightarrow R$ the morphism of right $R$ modules $\varphi_{x}(r_{1},\dots,r_{n})=\sum\limits_{i=1}^{n}x_{i}r_{i}$ and by $\psi_{x}$ the morphism of left modules defined by $\psi_{x}(r_{1},\dots,r_{n})=\sum\limits_{i=1}^{n}r_{i}x_{i},\,\forall \, (r_{1},\dots,r_{n})\in R^{n}$. Also because of the isomorphism $(R^{n})^{*}\simeq {}_{R}R^{n}, \, x\mapsto \varphi_{x}$, we will denote by $I^{\perp}=\{x\in R^{n}\mid \varphi_{x}(r)=0,\,\forall \, r\in I\}$ if $I$ is a right submodule of $R^{n}$ and similarly for left submodules $X$ of $R^{n}$, $X^{\perp}=\{x\in R^{n}\mid \psi_{x}(r)=0,\,\forall \, r\in X\}$. \\
Over a vector space $V$ there is an anti isomorphism of lattices between the lattice of closed subspaces of $V^{*}$ and the subspaces of $V$ given by $X\mapsto X^{\perp},\,\forall\,X\leq V$. We have the obvious
\begin{prop}\label{f.1}
For a right module $M$ the following are equivalent:\\
(i) The applications $M\geq X\mapsto X^{\perp}\leq M^{*}$ and $M^{*}\geq Y\mapsto Y^{\perp}\leq M$ between the lattice of the submodules of $M$ and the lattice of the closed submodules of $M^{*}$ are inverse anti isomorphism of lattices.\\
(ii) $(X^{\perp})^{\perp}=X,\,\forall\,X\leq M$ and $(Y^{\perp})^{\perp}=\overline{Y},\,\forall\,Y\leq M^{*}$.\\
(iii) $(X^{\perp})^{\perp}=X,\,\forall\,X\leq M$ and $(Y^{\perp})^{\perp}=Y,\,\forall\,Y\leq M^{*}$, $Y$ closed. \\
(iv) The applications of (i) are inverse to each other.
\end{prop}

\noindent
If $F$ is a finitely generated right $R$ module then every submodule of $F^{*}$ is closed, as if $Y$ is a left submodule of $F^{*}$ and $f\in \overline{Y}$, taking $\{x_{1},\dots,x_{n}\}$ the a system of generators of $F$, there is $g\in Y$ such that $g(x_{i})=f(x_{i})$, for all $i$, so $f=g\in Y$. Also it is easy to see that $R^{n}$ has orthogonal equivalence as right module if and only if it has orthogonal equivalence as left module, and this is equivalent to $(I^{\perp})^{\perp}=I,\,\forall\,I\leq R^{n}_{R}$ and $(X^{\perp})^{\perp}=X,\,\forall\,X\leq {}_{R}R^{n}$.

\begin{defin}
We will say that a right $R$ module $M$ has orthogonal equivalence (or orthogonal isomorphism, or shortly $M$ has $\perp$ equivalence) if the equivalent statements of Proposition \ref{f.1} hold. The ring $R$ will be called with $\perp$ equivalence if $R_{R}$ (or equivalently ${}_{R}R$) is a module with orthogonal equivalence.
\end{defin}

\begin{prop}\label{f.3}
Let $M$ be a right $R$ module and $X$ a submodule of $M$. Then we have the exact sequence 
$$0\longrightarrow (0^{\perp})^{\perp}\longrightarrow M\stackrel{\Phi_{M}}{\longrightarrow}{}^{*}(M^{*})$$
\end{prop}
\begin{proof}
For $x\in M$ we have $\Phi_{M}(x)=0\Leftrightarrow f(x)=0,\,\forall\,f\in M^{*}$ and this equivalent to $x\in (M^{*})^{\perp}=(0^{\perp})^{\perp}$, thus ${\rm ker }\,\Phi_{M}=(0^{\perp})^{\perp}$.
\end{proof}

\begin{prop}\label{f.4}
(i) For an $R$ module $M$ we have $(0^{\perp})^{\perp}=0$ if and only if $M$ is $R$ cogenerated, i.e. there is a monomorphism $M\hookrightarrow R^{I}$ for some set $I$. \\
(ii) If $\cal C$ is a class of right $R$ modules which is closed under quotients then the following are equivalent:\\
(a) $(X^{\perp})^{\perp}=X$ for all $M$ \,in \,${\cal C}$, $X<M$.\\
(b) $(0^{\perp})^{\perp}=0$ for all $M$ in $\cal C$.\\
(c) Any $M\in {\cal C}$ is cogenerated by $R$.\\
(d) $\Phi_{M}$ is a monomorphism for every $M$ in $\cal C$.
\end{prop}
\begin{proof}
(i) If $(0^{\perp})^{\perp}=0$ then take $I=M^{*}$ and $M\stackrel{i}{\longrightarrow}R^{I},\,i(x)=(f(x))_{f\in I}$; then of course $i$ is a monomorphism as $i(x)=0$ if and only if $f(x)=0,\,\forall f\in I=M^{*}$ i.e. $x\in (0^{\perp})^{\perp}=0$. Conversely, given a monomorphism $M\stackrel{i}{\hookrightarrow}R^{I}$, taking $\pi_{j}$ the canonical projections for all $j\in I$, we obtain the morphisms $f_{j}=\pi_{j}\circ i\in M^{*}$ and then $x\in (0^{\perp})^{\perp}=(M^{*})^{\perp}$ implies $f_{j}(x)=0,\,\forall\,j\in I$, i.e. $i(x)=0$ so $x=0$, as $i$ is injective. Thus $(0^{\perp})^{\perp}=0$. \\
(ii) (b) $\Leftrightarrow$ (c) by (i). (a) $\Leftrightarrow$ (b) follows as $\cal C$ is closed under quotient objects and denoting $\overline{0}$ the zero element of $M/X\in {\cal C}$ we have $(\{\overline{0}\}^{\perp})^{\perp}=(X^{\perp})^{\perp}$ from Lemma \ref{2}. Equivalence with (d) follows from Proposition \ref{f.3}
\end{proof}

\begin{prop}\label{f.5}
Suppose $R_{R}$ is a module with $\perp$ equivalence. Then $R$ contains all left simple modules and all right simple modules (up to an isomorphism; this is called a right - and left- Kasch ring).
\end{prop}
\begin{proof}
It is easy to see that for every right ideal $I$ of $R$ we have the isomorphism of left $R$ modules $(\frac{R}{I})^{*}\simeq I^{\perp}$, given by $I^{\perp}\ni f\mapsto f\circ \pi\in (\frac{R}{I})^{*}$, with $\pi:R\longrightarrow R/I$ the canonical projection. Then if $S$ is simple right module there is a maximal right ideal $M<R$ and an isomorphism $S\simeq \frac{R}{M}$. Then $S^{*}\simeq (\frac{R}{M})^{*}\simeq M^{\perp}\neq 0$ because if $M^{\perp}=0$ then $M=(M^{\perp})^{\perp}=0^{\perp}=R$, which contradicts the maximality of $M$. In a similar way one can see that $R$ contains all the isomorphism types of left $R$ modules.
\end{proof}

\noindent
We shall say a right (or left) $R$ module is $n$ generated if it has a system of $n$ generators.

\begin{lemma}\label{f.6}
Let $X$ be a right $R$ module such that every monomorphism $i:X\hookrightarrow M$ with the property that $M/{\rm Im}\,i$ is 1-generated splits. Then $X$ is an injective module.
\end{lemma}
\begin{proof}
Let $M$ be a right $R$ module such that $X<M$ (we identify $X$ with its image in $M$) and suppose $X\neq M$. Let ${\cal L}=\{Y<M\mid Y\neq 0 \,{\rm and}\,X\cap Y=0\}$. Then ${\cal L}\neq \emptyset$, because if $x\in M\setminus X$ then as $(X+xR)/X\neq 0$ is finitely generated then the hypothesis shows that there is $Y<X+xR$ such that $X\stackrel{.}{+}Y=X+xR$ and then $Y\neq 0$ as $x\notin X$, so $Y\in {\cal L}$. We can easily see that ${\cal L}$ is inductive, because if $(Y_{i})_{i\in I}$ is a totally ordered family of elements of ${\cal L}$ then $\bigcup\limits_{i\in I}Y_{i}$ is its majorant in ${\cal L}$. Take $N$ a maximal element of ${\cal L}$ and suppose $X+N\neq M$. Then there is $x\in M\setminus (X+N)$ and as $(X+N+xR)/(X+N)$ is finitely generated, by the hypothesis we can find $Y<M$ such that $X+N+Y=X+N+xR$ and $(X+N)\cap Y=0$. An easy computation shows now that $(N+Y)\cap X=0$ and so $N+Y=N$ by the maximality of $N$. Thus we obtain $X+N+Y=X+N=X+N+xR$ which is a contradiction, because $x\notin X+N$. We find that $X$ is a direct summand in $M$ for every module $M$ such that $X\hookrightarrow M$, so $X$ is injective in ${\cal M}_{R}$.
\end{proof}

\begin{prop}\label{f.7}
Let $R$ be a ring with $\perp$ equivalence. If $R\stackrel{j}{\hookrightarrow} X$ is a monomorphism of right (left) $R$ modules and $X$ is $R$ cogenerated then $j$ splits.
\end{prop}
\begin{proof}
Consider $X\stackrel{\sigma}{\hookrightarrow}R^{I}$ a monomorphism and let $(x_{i})_{i\in I}=\sigma(j(1))$. Then we have $(x_{i}r)_{i\in I}=\sigma(j(1))r=\sigma(j(r))$ and as $j,\sigma$ are injective we see that $x_{i}r=0,\,\forall\,i\in I$ if and only if $r=0$. This shows that $\bigcap\limits_{i\in I}Rx_{i}^{\perp}=0$. Then we have $0=\bigcap\limits_{i\in I}Rx_{i}^{\perp}=(\sum\limits_{i\in I}Rx_{i})^{\perp}$ (by Proposition \ref{3}), so $\sum\limits_{i\in I}Rx_{i}=((\sum\limits_{i\in I}Rx_{i})^{\perp})^{\perp}=0^{\perp}=R$. Then we find that there is $F$ a finite subset of $I$ such that $\sum\limits_{i\in F}Rx_{i}=R$, thus there are $(y_{i})_{i\in F}\in R$ such that $\sum\limits_{i\in F}y_{i}x_{i}=1$. Now if we denote by $\pi_{F}$ the projection of $R^{I}$ on $R^{F}$, $\pi_{F}((r_{i})_{i\in I})=(r_{i})_{i\in F}$ and by $y=(y_{i})_{i\in F}\in R^{F}=R^{(F)}$, then $\varphi_{y}(\pi_{F}(\sigma(j(r))))=\varphi_{y}(\pi_{F}((x_{i}r)_{i\in I}))=\varphi_{y}((x_{i}r)_{i\in F})=\sum\limits_{i\in F}y_{i}x_{i}r=r$, so $\varphi_{y}\circ\pi_{F}\circ\sigma\circ j=id_{R}$, showing that the morphism of right modules $\varphi_{y}\circ\pi_{F}\circ\sigma:X\longrightarrow R$ is a split for $j$.
\end{proof}

\begin{lemma}\label{f.8}
$R^{n}$ has orthogonal equivalence (as left or right $R$ module) if and only if every $n$ generated right (or left) module has orthogonal equivalence.
\end{lemma}
\begin{proof}
Suppose $R^{n}$ has $\perp$ equivalence. Let $F=R^{n}/X$ be a right $n$ generated $R$ modules and $\pi:R^{n}\longrightarrow F$ the canonical projection. For each $g\in X^{\perp}$ ($X<R^{n}$) we denote by $\overline{g}\in F^{*}$ the (unique) morphism for which $\overline{g}\circ\pi=g$ and with $\hat{x}=\pi(x)$ - the class of an element $x\in R^{n}$. Now we see that if $Y<F^{*}$ and $Z=\{\alpha\circ\pi\mid\alpha\in Y\}$, then $Y=\{\overline{g}\mid g\in Z\}$, $Y^{\perp}=\{\hat{x}\mid \overline{g}(\hat{x})=0,\,\forall\,g\in Z\}=Z^{\perp}/X$ ($Z\subseteq X^{\perp}$ so $Z^{\perp}\supseteq (X^{\perp})^{\perp}=X$) and $(Y^{\perp})^{\perp}=\{\overline{g}\mid\overline{g}(\hat{x})=0,\,\forall\,\hat{x}\in Z^{\perp}/X\}=\{\overline{g}\mid g(x)=0,\,\forall\,x\in Z^{\perp}\}=\{\overline{g}\mid g\in (Z^{\perp})^{\perp}=Z\}=Y$.\\
Now if $Y<F$ and $Z=\pi^{-1}(Y)$ then $Y^{\perp}=\{\overline{g}\mid\overline{g}(\hat{x})=0,\,\forall\,\hat{x}\in Y\}=\{\overline{g}\mid g(x)=0,\,\forall\,x\in Z\}=\{\overline{g}\mid g\in Z^{\perp}\}$ and $(Y^{\perp})^{\perp}=\{\hat{x}\mid\overline{g}(\hat{x})=g(x)=0,\,\forall\,g\in Z^{\perp}\}=\{\hat{x}\mid x\in (Z^{\perp})^{\perp}=Z\}=Y$.
\end{proof}

\begin{th} The following assertions are equivalent: \\
(i) Every right $R$ module has $\perp$ equivalence. \\
(ii) Every finitely generated module has $\perp$ equivalence. \\
(iii) Every left $R$ module has $\perp$ equivalence. \\
(iv) Every finitely generated module has $\perp$ equivalence.\\
(v) $R$ is a PF ring (both left and right).\\
(vi) $(X^{\perp})^{\perp}=X$ for all $X<M$ in ${\cal M}_{R}$ or in ${}_{R}{\cal M}$.\\
(vii) $R^{2}$ has $\perp$ equivalence.
\end{th}
\begin{proof}
$\bullet$ (v) $\Rightarrow$ (i) and (v) $\Rightarrow$ (vi) follow from Lemma \ref{2} so we have the implications (v) $\Rightarrow$ (i) $\Rightarrow$ (ii) $\Rightarrow$ (vii) and (v) $\Rightarrow$ (vi) $\Rightarrow$ (vii)\\
$\bullet$ (v) $\Rightarrow$ (iii) $\Rightarrow$ (iv) $\Rightarrow$ (vii) is the left symmetric of (v) $\Rightarrow$ (i) $\Rightarrow$ (ii) $\Rightarrow$ (vii). \\
$\bullet$ (vii) $\Rightarrow$ (v) If $R^{2}$ has $\perp$ equivalence, then by Lemma \ref{f.8} we have that any 2 generated right (and any left) module has $\perp$ equivalence, in particular $R$ has orthogonal equivalence. Now let $R\stackrel{i}{\hookrightarrow}X$ be a monomorphism in ${\cal M}_{R}$ such that $X/i(R)$ is 1 generated. Then as $X$ has $\perp$ equivalence, Proposition \ref{f.4} shows that $X$ is $R$ cogenerated as right $R$ module. Now by Proposition \ref{f.7} $i$ splits, as $X$ is $R$ cogenerated and $R$ has $\perp$ equivalence. Then we can apply Lemma \ref {f.6} and obtain that $R_{R}$ is injective. Because $R$ has $\perp$ equivalence, by Proposition \ref{f.5} we obtain that $R_{R}$ contains all isomorphism types of simple right modules, and as $R_{R}$ is injective, we obtain that $R_{R}$ is an injective cogenerator of ${\cal M}_{R}$, i.e. a right RF ring. Similarly we can show that $R$ is also a left PF ring.
\end{proof}

\begin{cor}
If $R$ is a PF ring, then $F\simeq {}^{*}(F^{*})$ by $\Phi_{F}$ for every finitely generated left module (the analogue holds for right modules).
\end{cor}
\begin{proof}
Proposition \ref{f.3} shows that $\Phi_{F}$ is injective. By the same argument as in the proof of Lemma \ref{2} we have that ${}_{R}R$ injective implies that $\Phi_{F}$ is an epimorphism and the conclusion is proved.
\end{proof}

\begin{cor}
$R$ is a PF ring if and only if for every finitely generated right (or left) $R$ module $F$, the lattice of the submodules of $F$ is anti isomorphic to the lattice of the submodules of $F^{*}$ via the $\perp$ applications of Proposition \ref{f.1}, equivalently, the dual lattice of the submodules of any finitely generated right module is isomorphic (via $\perp$ applications) to the lattice of the submodules of the dual of that module.
\end{cor}

\end{section}

\end{document}